\documentclass[12pt,a4paper]{article}
\usepackage{a4wide}

\usepackage[T1]{fontenc}
\usepackage{microtype}
\usepackage{amsthm,amsmath,amssymb}
\usepackage{enumitem}
\usepackage{graphicx}
\usepackage[pdftitle={Induced Ramsey-type results and binary predicates for point sets}, pdfauthor={M. Balko, J. Kyncl, A. Pilz, S. Langerman}]{hyperref}
\usepackage{lmodern}

\usepackage{color}

\definecolor{modra3}{rgb}{.1,.0,.4}
\hypersetup{colorlinks=true, linkcolor=modra3, urlcolor=modra3, citecolor=modra3, pdfpagemode=UseNone, pdfstartview=} 

\def\Pr{\mathbf{P}}
\def\inst#1{$^{#1}$}

\newtheorem{theorem}{Theorem}
\newtheorem{lemma}[theorem]{Lemma}
\newtheorem{corollary}[theorem]{Corollary}
\newtheorem{proposition}[theorem]{Proposition}

\newtheorem{observation}[theorem]{Observation}
\newtheorem{problem}[theorem]{Problem}

\begin{document}

\title{Induced Ramsey-type results and binary predicates for point sets\thanks{The research for this article was partially carried out in the course of the bilateral research project ``Erd\H{o}s--Szekeres type questions for point sets'' between Graz and Prague, supported by the OEAD project CZ~18/2015 and project no.\ 7AMB15A~T023 of the Ministry of Education of the Czech Republic.
The first two authors were partially supported by the project CE-ITI no.\ P202/12/G061 of the Czech Science Foundation (GA\v CR).
The first author was also partially supported by the grant GAUK 690214 and has received funding from European Research Council (ERC) under the European Union's Horizon 2020 research and innovation programme under grant agreement No. 678765.
The third author is Directeur de Recherches du F.R.S.-FNRS.
The fourth author is supported by a Schr\"odinger fellowship, Austrian Science Fund (FWF): J-3847-N35.}
}

\author{Martin Balko\inst{1,2} 
\and
Jan Kyn\v{c}l\inst{1} 
\and Stefan Langerman\inst{3} 
\and
Alexander Pilz\inst{4} 
}

\maketitle

\begin{center}
{\footnotesize
\inst{1} 
Department of Applied Mathematics and Institute for Theoretical Computer Science, \\
Faculty of Mathematics and Physics, Charles University, Czech Republic \\
\texttt{balko@kam.mff.cuni.cz, kyncl@kam.mff.cuni.cz}
\\\ \\
\inst{2} 
Department of Computer Science, Faculty of Natural Sciences, Ben-Gurion University
of the Negev, Beer~Sheva, Israel
\\\ \\
\inst{3}
D\'{e}partement d'Informatique Universit\'{e} Libre de Bruxelles, Brussels, Belgium\\
\texttt{stefan.langerman@ulb.ac.be}
\\ \ \\
\inst{4}
Department of Computer Science, ETH Z\"{u}rich, Zurich, Switzerland \\
\texttt{alexander.pilz@inf.ethz.ch}
}
\end{center}

\begin{abstract}
Let $k$ and $p$ be positive integers and let $Q$ be a finite point set in general position in the plane.
We say that $Q$ is \emph{$(k,p)$-Ramsey} if there is a finite point set $P$ such that for every $k$-coloring $c$ of $\binom{P}{p}$ there is a subset $Q'$ of $P$ such that $Q'$ and $Q$ have the same order type and $\binom{Q'}{p}$ is monochromatic in $c$.
Ne\v{s}et\v{r}il and Valtr proved that for every $k \in \mathbb{N}$, all point sets are $(k,1)$-Ramsey. They also proved that for every $k \ge 2$ and $p \ge 2$, there are point sets that are not $(k,p)$-Ramsey.

As our main result, we introduce a new family of $(k,2)$-Ramsey point sets, extending a result of Ne\v{s}et\v{r}il and Valtr.
We then use this new result to show that for every $k$ there is a point set $P$ such that no function $\Gamma$ that maps ordered pairs of distinct points from $P$ to a set of size $k$ can satisfy the following ``local consistency'' property: if $\Gamma$ attains the same values on two ordered triples of points from $P$, then these triples have the same orientation.
Intuitively, this implies that there cannot be such a function that is defined locally and determines the orientation of point triples.
\end{abstract}

%========================================================================
\section{Introduction}
\label{sec-introduction}

In this paper, we study induced Ramsey-type results for point sets and we present an application of our results to the problem of encoding sets of points with binary functions.

Let $k$ be a positive integer and let $X$ be a set, not necessarily finite.
A \emph{$k$-coloring} of~$X$ is a function $c \colon X \to C$ where $C$ is a set of size $k$.
We call the elements of $C$ \emph{colors} and we say that a subset $Y$ of $X$ is \emph{monochromatic in $c$} if all the elements of $Y$ have the same color in $c$.
Let $p$ be a positive integer. 
We use $\binom{X}{p}$ to denote the set of all $p$-element subsets (equivalently, unordered $p$-tuples of distinct elements) of $X$ and $(X)_p$ to denote the set of all ordered $p$-tuples of distinct elements of $X$. 
We use $[p]$ to denote the set $\{1,2,\dots,p\}$.

Let $P$ and $Q$ be two finite sets of points in the plane in \emph{general position};
that is, neither of these sets contains three points on a common line.
The \emph{order-type function of $P$} is the function $\Delta_P \colon (P)_3 \to \{-1,1\}$ where $\Delta_P(a,b,c)=1$ if the triple $(a,b,c)$ traced in this order is oriented counterclockwise and $\Delta_P(a,b,c)=-1$ otherwise.
By an \emph{order type} we mean an equivalence class of point sets under the following notion of isomorphism.
We say that $P$ and $Q$ \emph{have the same order type} if there is a one-to-one correspondence $f \colon P \to Q$ such that every ordered triple of points of $P$ has the same orientation (either clockwise or counterclockwise) as its image via~$f$.
A point set is in \emph{convex position} if its points are vertices of a convex polygon.
For two points $u$ and $v$ in the plane, we use $\overline{uv}$ to denote the line determined by $u$ and $v$ directed from $u$ to $v$.
We let $x(u)$ be the $x$\nobreakdash-coordinate of a point $u \in \mathbb{R}^2$.

Unless stated otherwise, we assume that every considered set $P$ of points is planar, finite, in general position, and that the $x$\nobreakdash-coordinates of points from $P$ are distinct.

Let $k$ and $p$ be positive integers and let $P$ and $Q$ be two point sets.
We use the standard arrow notation $P \to (Q)^p_k$ to abbreviate the following statement:
for every $k$-coloring $c$ of $\binom{P}{p}$ there is a subset $Q'$ of $P$ such that $Q'$ and $Q$ have the same order type and $\binom{Q'}{p}$ is monochromatic in $c$.
If there is a point set $P$ such that $P \to (Q)^p_k$, then we say that $Q$ is \emph{$(k,p)$-Ramsey}.

We focus on the problem of characterizing $(k,2)$-Ramsey point sets, which is currently wide open. For $p=1$ and $p\ge 3$, $(k,p)$-Ramsey sets have been completely characterized; see Subsections~\ref{subsec-introRamsey} and~\ref{subsec-finalRamseyOrderType}.

As our main result, we introduce a new family of $(k,2)$-Ramsey point sets.
We then use this result to show that for every $k$ there is a point set $P$ such that no function $\Gamma$ that maps ordered pairs of distinct points from $P$ to a set of size $k$ can satisfy the following ``local consistency'' property: if $\Gamma$ attains the same values on two ordered triples of points from $P$, then these triples have the same orientation.

%-------------------------------------------------------------------------
\subsection{Induced Ramsey-type results for point sets}
\label{subsec-introRamsey}

The problem of determining which point sets are $(k,p)$-Ramsey has already been considered in the literature~\cite{nesSolVal99,nesetrilValtr94,nesetrilValtr98}.
Ne\v{s}et\v{r}il and Valtr~\cite{nesetrilValtr94} showed the following result.

\begin{theorem}[{\cite[Theorem~3]{nesetrilValtr94}}]
\label{thm-nesValtr}
Let $Q$ be a finite set of points in the plane and let $k \ge 2$ be an integer.
There is a finite set $P=P(Q,k)$ of points in the plane such that for every $k$-coloring $c$ of points from~$P$ there is a subset $Q'$ of $P$ with the following three properties:
\begin{enumerate}[label=(\roman*)]
\item\label{item-nesVal1}$Q'$ is monochromatic in $c$,
\item\label{item-nesVal2}$Q$ and $Q'$ have the same order type, and
\item\label{item-nesVal3}the convex hull of $Q'$ does not contain any points from $P \setminus Q'$.
\end{enumerate}
\end{theorem}

Clearly, for all positive integers $k$ and $p$, every point set is $(1,p)$-Ramsey and every set with at most $p$ points is $(k,p)$-Ramsey.
Theorem~\ref{thm-nesValtr} implies that every point set is $(k,1)$-Ramsey for every integer $k \ge 2$.
In fact, Theorem~\ref{thm-nesValtr} is even stronger because of part~\ref{item-nesVal3}.

Ne\v{s}et\v{r}il and Valtr~\cite{nesetrilValtr94,nesetrilValtr98} also proved that for all $k \ge 2$ and $p \ge 2$ there are point sets that are not $(k,p)$-Ramsey.
In particular, they showed that for every integer $p \ge 2$ there exists a point set $Q=Q(p)$ and a 2-coloring $c$ of $\binom{\mathbb{R}^2}{p}$ such that no subset $R\subset \mathbb{R}^2$ with monochromatic $\binom{R}{p}$ in $c$ has the same order type as $Q$~\cite[Theorem~5]{nesetrilValtr94}.

On the other hand, some point sets are $(k,2)$-Ramsey for $k \ge 2$.
Ne\v{s}et\v{r}il and Valtr~\cite{nesetrilValtr94} showed that every set of four points not in convex position is $(k,2)$-Ramsey for every positive integer $k$~\cite[Theorem~6]{nesetrilValtr94}. 

The problem of determining whether a point set is $(k,2)$-Ramsey has the following equivalent formulation.
For a given point set $Q$ and $k \in \mathbb{N}$, is there a point set $P=P(Q,k)$ such that in every $k$-coloring of the edges of the complete geometric graph $K_P$ on $P$ there is a monochromatic complete subgraph of $K_P$ such that its vertex set has the same order type as $Q$?

%-----------------------------------------------------------------------
\subsection{Predicates for order types}
\label{subsec-introPredicate}

Let $\mathcal{P}$ be the family of all finite planar point sets in general position with distinct $x$\nobreakdash-coordinates and let $Z$ be some  finite set.
For $t \in \mathbb{N}$, a function $\Gamma \colon \{(P,T): P \in \mathcal{P}, T \in (P)_t \} \to Z$ is a \emph{$t$-ary point-set predicate} with \emph{codomain} $Z$.
For $P \in \mathcal{P}$, we let $\Gamma_P \colon (P)_t \to Z$ be the restriction of $\Gamma$ to $P$; more precisely, we define $\Gamma_P(T)$ as $\Gamma(P,T)$.
In the cases $t=2$ and $t=3$ we just say that $\Gamma$ is a \emph{binary} and \emph{ternary} point-set predicate, respectively.
We sometimes also shorten the term ``point-set predicate'' to ``predicate''.

An example of a ternary point-set predicate is the function $\Delta$ whose restriction to every $P$ from $\mathcal{P}$ is the order-type function $\Delta_P$.
Note that if $Q \subseteq P$, then $\Delta_Q = \Delta_P \restriction (Q)_3$ and thus $\Delta$ can be considered as a mapping from $(\mathbb{R}^2)_3$ to $\{-1,1\}$.
However, this might not be the case for all point-set predicates, since their definition is quite general; see Section~\ref{sec-FelsnerPredicate} for an example.

Clearly, the predicate $\Delta$ gives the upper bound $2^{O(n^3)}$ on the number of different order types of sets with $n$ points.
This bound is far from being tight as it is well-known that the number of different order types of sets with $n$ points is in $2^{\Theta(n\log{n})}$~\cite{alon86,goodmanPollack86}.

Considering \emph{generalized point sets}, where each pair of points lies on an $x$-monotone curve such that these curves form an arrangement of pseudolines (see~\cite{felsner97} for definitions), one can define the \emph{order type} of a generalized point set (sometimes called an \emph{abstract order type}) analogously as for point sets.
We note that order types of generalized point sets are relabeling classes of \emph{uniform acyclic oriented matroids of rank 3}.
The number of different order types of generalized point sets with $n$ points is in $2^{\Theta(n^2)}$~\cite{goodmanPollack83,knuth92}.

These upper bounds suggest the question whether the order type of a point set or a generalized point set can be encoded by a binary predicate.
Let $\mathcal{C}\subseteq \mathcal{P}$ be a class of point sets.
We say that a $t$-ary predicate $\Gamma$ \emph{encodes the order types} of sets from $\mathcal{C}$ if whenever there is a one-to-one correspondence $f \colon P \to Q$ between two sets from $\mathcal{C}$ such that $\Gamma_P(p_1,\dots,p_t)=\Gamma_Q(f(p_1),\dots,f(p_t))$ for every $t$-tuple $(p_1,\dots,p_t) \in (P)_t$, then $\Delta_P(a,b,c)=\Delta_Q(f(a),f(b),f(c))$ for every triple $(a,b,c)\in (P)_3$.

It is indeed possible to devise such a predicate: one can show its existence with a probabilistic argument, using the fact that the number of automorphism-free functions $f \colon ([n])_2 \to [2]$ exceeds the number of order types of $n$-point sets.
Moreover, in Section~\ref{sec-FelsnerPredicate}, we use a result of Felsner~\cite{felsner97} to explicitly construct a binary predicate $\Psi$ with codomain $\{0,1\}$ that encodes the order types of all point sets from $\mathcal{P}$.
Of course, the predicate $\Psi$ has a certain drawback.
Unlike the predicate $\Delta$, the predicate $\Psi$ does not behave well ``locally''.
In particular, we would like to keep the property that if $\Psi_P$ attains the same values on two ordered triples of points from $P$, then these triples have the same orientation.
The predicate $\Psi$ does not fulfill this property.

To capture the demand on local behavior of a binary predicate $\Gamma$, we introduce the following definition.
We say that $\Gamma$ is \emph{locally consistent on a set $P\in \mathcal{P}$} if, for any two distinct subsets $\{a_1,a_2,a_3\}$ and $\{b_1,b_2,b_3\}$ of $P$, having $\Gamma_P(a_i,a_j)=\Gamma_P(b_i,b_j)$ for all distinct $i$ and $j$ from $\{1,2,3\}$ implies $\Delta_P(a_1,a_2,a_3)=\Delta_P(b_1,b_2,b_3)$.
If a binary predicate $\Gamma$ is locally consistent on all sets from a class $\mathcal{C} \subseteq \mathcal{P}$, then we say $\Gamma$ is \emph{locally consistent on~$\mathcal{C}$}.
If $\mathcal{C}=\mathcal{P}$, then we just say that $\Gamma$ is \emph{locally consistent}.

The following question was the main motivation for our research.

\begin{problem}
\label{prob-binaryHereditaryPredicate}
Is there a locally consistent binary predicate that encodes order types of all sets from~$\mathcal{P}$?
\end{problem}

\subsubsection{Known point-set predicates}

Several predicates that encode order types of all sets from $\mathcal{P}$ are known, but none of them is binary and locally consistent. The predicate $\Delta$ is the ``default'' such predicate.
Similar predicates naturally occur in the investigation of combinatorial properties of point sets, and can be obtained from various combinatorial structures.
Several of these predicates have interesting applications.
For example, Felsner~\cite{felsner97} shows how to encode the order type in an $n \times n$ $\{0,1\}$-matrix and uses the resulting binary predicate to estimate the number of arrangements of $n$ pseudolines.
Here, we give examples of other point-set predicates that have been studied in the past.

The partial order of the slopes of the lines spanned by all pairs of points of a point set $P$ determines the order type of $P$ and also the \emph{circular sequence of permutations} of $P$, which are obtained by projecting $P$ to a line in all possible directions~\cite{semispaces}. The order of the slopes of the lines can be obtained from a 4-ary predicate $\Xi$ with $\Xi(a,b,c,d) = 1$ if the supporting line of $ab$ has smaller slope than the one of $cd$, and $\Xi(a,b,c,d) = -1$ otherwise.

Goodman and Pollack~\cite{semispaces} showed that up to the mirror symmetry, the order type of~$P$ is determined by the family of all intersections of $P$ with halfplanes.
A 4-ary predicate with codomain $\{-1,1\}$ indicating whether two points are on the same side of a line defined by two other points provides this information.

For some known predicates, additional information on the extreme points is required.
There is a predicate similar to but less powerful than the one identifying the semispaces of~$P$ that is implied by the work of Adaricheva and Wild~\cite{adarichevaWild10} on convex geometries; in our terminology, given the convex hull of~$P$, knowing whether a point of $P$ is inside the convex hull of three others determines the order type of~$P$. 
Aichholzer et al.~\cite{acklv14} show that the radial order in which the points of $P \setminus \{p\}$ appear around each point~$p \in P$ determines the order type of $P$ if it has at least four extreme points, or if the extreme points are known.
(There are point sets with triangular convex hull and different order types but the same radial orders at corresponding points.)
A 4-ary predicate with codomain $\{-1, 1\}$ can provide these radial orders.
They are also determined by the set of crossing edge pairs in the complete geometric graph on~$P$~\cite{kyncl_realizability}, giving yet another 4-ary predicate.

While the reader may easily come up with further ternary or 4-ary predicates that encode order types of all point sets from $\mathcal{P}$, we are interested in the existence of binary point-set predicates that are locally consistent and that encode order types of all point sets, since encoding order types with triple-orientations seems highly inefficient with respect to the amount of space needed.

%========================================================================
\section{Our results}
\label{sec-results}

For a line $\overline{uv}$ with $x(u) < x(v)$ and a point $w \in \mathbb{R}^2$, we say that $w$ is \emph{above} $\overline{uv}$ if $(u,v,w)$ is oriented counterclockwise.
Similarly, $w$ is  \emph{below} $\overline{uv}$ if $(u,v,w)$ is oriented clockwise.
Let $A=\{a_1,\dots,a_{|A|}\}$ and $B=\{b_1,\dots,b_{|B|}\}$ be two point sets with  $x(a_1) <  \cdots < x(a_{|A|})$ and $x(b_1) < \cdots < x(b_{|B|})$.
We say that $A$ lies \emph{deep below} $B$ if every point from $B$ lies above every line $\overline{a_ia_j}$ with $i<j$ and every point from $A$ lies below every line $\overline{b_ib_j}$ with $i<j$.
If $x(a_{|A|}) < x(b_1)$, then we write $x(A) < x(B)$.

We say that a point set $P$ is \emph{decomposable} if either $|P|=1$ or there is a partition $P_1 \cup P_2$ of $P$ that satisfies the following conditions:
\begin{enumerate}[label=(\roman*)]
\item\label{item-decomposable1} both point sets $P_1$ and $P_2$ are nonempty and $x(P_1)<x(P_2)$,
\item\label{item-decomposable3} $P_1$ is deep below $P_2$, and
\item\label{item-decomposable2} both point sets $P_1$ and $P_2$ are decomposable.
\end{enumerate}
The class of decomposable sets includes, for example, sets constructed by Erd\H{o}s and Szekeres~\cite{erdosSzekeres1935} in their proof of the lower bound in the Erd\H{o}s--Szekeres Theorem.

If the partition $P_1 \cup P_2$ of $P$ satisfies conditions~\ref{item-decomposable1} and \ref{item-decomposable3} (and not necessarily condition~\ref{item-decomposable2}), then we say that $P_1 \cup P_2$ is a \emph{splitting} of $P$.

First, we extend the result of Ne\v{s}et\v{r}il and Valtr~\cite[Theorem~6]{nesetrilValtr94} as follows.

\begin{theorem}
\label{thm-decomposableRamsey}
For every positive integer $k$, every decomposable set is $(k,2)$-Ramsey.
\end{theorem}

We further show that, for $k \ge 2$ and $p \ge 3$, $(k,p)$-Ramsey sets are exactly point sets in convex position (Proposition~\ref{prop-RamseyCharacterization}).
We also present a short proof of the fact that for any positive integer $k$ every point set is $(k,1)$-Ramsey (Lemma~\ref{lem-signatures}).
An interesting problem, inspired by a question of a reviewer, is whether there is a $(2,2)$-Ramsey point set that is not decomposable. 

Our study of $(k,p)$-Ramsey sets was motivated by questions about binary point-set predicates.
In particular, Problem~\ref{prob-binaryHereditaryPredicate} was our main motivation.

Using a result of Felsner~\cite{felsner97}, we find a binary point-set predicate $\Psi$ with codomain $\{0,1\}$ that encodes order types of all point sets from $\mathcal{P}$; see Section~\ref{sec-FelsnerPredicate}.
As already mentioned, the predicate $\Psi$ is not locally consistent. 
In fact, using Theorem~\ref{thm-decomposableRamsey}, we show that no binary point-set predicate is locally consistent.
This gives a negative solution to Problem~\ref{prob-binaryHereditaryPredicate}.

\begin{theorem}
\label{thm-noHereditaryPredicate}
For every finite set $Z$, there is a point set $P=P(|Z|)$ such that no binary point-set predicate with codomain $Z$ is locally consistent on~$P$. 
\end{theorem}

No binary predicate is locally consistent on all point sets, but there might be binary predicates that are locally consistent on more restricted classes of point sets and that encode order types of point sets from these classes.
As a first step in this direction, we find a binary predicate with codomain of size only 2 that is locally consistent on \emph{wheel sets}, that is, point sets $P$ with at least $|P|-1$ extremal points. 

Wheel sets have been studied, for example, in connection with combinatorially different simplicial polytopes with $n$ vertices in dimension $n-3$.
It follows from a result by Perles (see~\cite[Chapter~6.3]{gruenbaum_book}) that there are $\Theta(2^n/n)$ different order types of wheel sets of size $n$.
See~\cite{embracing} for further results on wheel sets and the historical background.

\begin{proposition}
\label{prop-hereditaryPredicate}
The order types of wheel sets can be encoded with a binary point-set predicate $\Phi$ with codomain $\{-1,1\}$ such that $\Phi$ is locally consistent on the class of all wheel sets.
\end{proposition}

Since there are only $\Theta(2^n/n)$ different order types of wheel sets of size $n$, the binary predicate from Proposition~\ref{prop-hereditaryPredicate} is ``inefficient'' in a similar way that the order type function is inefficient in encoding order types of all point sets.

We also try to estimate the growth rate of the function $h \colon \mathbb{N} \to \mathbb{N}$ where $h(k)$ is the largest integer such that there is a binary predicate with codomain of size $k$ that is locally consistent on all point sets of size $h(k)$ and that encodes their order types.

By Theorem~\ref{thm-noHereditaryPredicate}, we know that $h(k)$ is finite for every $k$ and thus well-defined.
On the other hand, we show that $h(k) \ge \Omega(k^{3/2})$.

\begin{theorem}
\label{thm-smallSets}
For every positive integer $k$, there is a binary point-set predicate with codomain of size $k$ that is locally consistent on all point sets of size at most $ck^{3/2}$ for some constant $c>0$ and that encodes their order types.
\end{theorem}

We prove Theorem~\ref{thm-decomposableRamsey} in Section~\ref{sec-decomposableRamsey}.
In Section~\ref{sec-FelsnerPredicate}, we give an example of a binary predicate that encodes order types of all point sets.
In Section~\ref{sec-hereditaryPredicate} we prove Proposition~\ref{prop-hereditaryPredicate} and also show that wheel sets are the only point sets with a locally consistent ``antisymmetric'' predicate with codomain $\{-1,1\}$.
Theorems~\ref{thm-noHereditaryPredicate} and~\ref{thm-smallSets} are  proved in Sections \ref{sec-noHereditaryPredicate} and~\ref{sec-smallSets}, respectively.
Finally, in Section~\ref{sec-final}, we discuss some open problems and possible directions for future research.

%========================================================================
\section{Proof of Theorem~\ref{thm-decomposableRamsey}}
\label{sec-decomposableRamsey}

Here we show that decomposable sets are $(k,2)$-Ramsey for every positive integer $k$.
That is, if $Q$ is a decomposable set, then there is a point set $P=P(Q,k)$ such that $P \to (Q)^2_k$.

We say that point sets $P$ and $Q$ \emph{have the same signature} if there is a one-to-one correspondence $f \colon P \to Q$ that preserves the triple-orientations and the total order of the $x$\nobreakdash-coordinates of the points.
Clearly, if $P$ and $Q$ have the same signature, then they have the same order type.
The converse is not true already for sets of three points, which have just one possible order type but two possible signatures.

Let $k$ be a positive integer and let $Q_1, \dots,Q_k$ be point sets.
For a point set $P$ and a positive integer $p$, we write $P \xrightarrow{{x}} (Q_1,\dots,Q_k)^p$ to denote the following statement: for every $k$-coloring $c$ of $\binom{P}{p}$ there is an $i \in [k]$ and a subset $Q'_i$ of $P$ that has the same signature as $Q_i$ and such that all $p$-tuples of points from $Q'_i$ have color $i$ in $c$.
If $Q_i$ and $Q$ have the same signature for every $i \in [k]$, then we write $P \xrightarrow{{x}} (Q)^p_k$.
Observe that if $P \xrightarrow{{x}} (Q)^p_k$, then $P \to (Q)^p_k$.
If there is a point set $P$ such that $P \xrightarrow{{x}} (Q)^p_k$, then we say that $Q$ is \emph{ordered $(k,p)$-Ramsey}. 

The following result implies that decomposable sets are ordered $(k,2)$-Ramsey.

\begin{theorem}
\label{thm-decomposableRamseyNondiagonal}
Let $k$ be a positive integer and let $Q_1, \dots,Q_k$ be decomposable point sets.
Then there is a point set $P=P(Q_1,\dots,Q_k)$ such that $P \xrightarrow{{x}} (Q_1,\dots,Q_k)^2$.
\end{theorem}

Theorem~\ref{thm-decomposableRamseyNondiagonal} immediately implies Theorem~\ref{thm-decomposableRamsey} by choosing $Q_i = Q$ for every $i \in [k]$.
In the proof of Theorem~\ref{thm-decomposableRamseyNondiagonal}, we need to use the following fact, which says that all point sets are ordered $(k,1)$-Ramsey.

\begin{lemma}
\label{lem-signatures}
Let $k$ be a positive integer and let $Q_1,\dots,Q_k$ be point sets.
Then there is a point set $P=P(Q_1,\dots,Q_k)$ such that $P \xrightarrow{{x}} (Q_1,\dots,Q_k)^1$.
\end{lemma}
\begin{proof}
For point sets $A$ and $B$ we let $A \circ B$ be a point set that is constructed as follows.
We replace every point $a$ from $A$ by a small neighborhood $N_a$ and we place a small scaled and translated copy of $B$ into each such neighborhood.
The neighborhoods are chosen to be small enough so that $x(N_a)<x(N_{a'})$ for all $a,a'$ from $A$ with $x(a)<x(a')$ and so that no line intersects three of these neighborhoods.
It is easy to see that the operation $\circ$ is associative if we do not distinguish point sets with the same signatures.

We show by induction on $k$ that $Q_1 \circ \dots \circ Q_k \xrightarrow{{x}} (Q_1,\dots,Q_k)^1$.
The statement is trivial for $k=1$, so we assume that $k \ge 2$.
Let $c$ be a $k$-coloring of the points of $Q_1 \circ \dots \circ Q_k$.
If there is a point of color $1$ in every neighborhood $N_q$ for $q \in Q_1$, then we have $Q'$ such that $Q'$ and $Q_1$ have the same signature and all points from $Q'$ have color $1$ in $c$.
So we assume that there is a neighborhood $N_q$ for some $q \in Q_1$ such that no point from $Q_1 \circ \dots \circ Q_k \cap N_q$ has color $1$ in $c$.
Then the set $(Q_1 \circ \dots \circ Q_k) \cap N_q$ is colored with colors $2,3,\dots,k$ and, since signatures are preserved by any scaling and translation, it has the same signature as the set $Q_2\circ \dots \circ Q_k$.
Thus, by the induction hypothesis, $(Q_1 \circ \dots \circ Q_k) \cap N_q \xrightarrow{{x}} (Q_2,\dots,Q_k)^1$, which finishes the proof.
\end{proof}

We also use the following result, which is similar to Lemma~10 in~\cite{nesetrilValtr94}.

\begin{lemma}
\label{lem-nesValtr}
Let $S$ be a point set and let $S_1 \cup S_2$ be a splitting of $S$.
Let $k$ be a positive integer.
There is a point set $R=R(S,k)$ with a splitting $R=R_1 \cup R_2$ that satisfies the following two properties:
\begin{enumerate}[label=(\roman*)]
\item\label{item-lemNesVal1}$R$ is in general position,
\item\label{item-lemNesVal2}for every $k$-coloring $c$ of $R_1 \times R_2$ there exists a subset $S'$ of $R$ such that $(S' \cap R_i)$ and $S_i$ have the same signature for both $i \in \{1,2\}$, and $(S'\cap R_1) \times (S' \cap R_2)$ is monochromatic in $c$.
\end{enumerate}
\end{lemma}
\begin{proof}
Let $R_1$ and $R_2$ be sets such that $R_1 \xrightarrow{{x}} (S_1)^1_k$ and $R_2 \xrightarrow{{x}} (S_2)^1_{k^{|R_1|}}$.
The sets $R_1$ and $R_2$ exist by Lemma~\ref{lem-signatures}.
We construct $R$ by translating $R_1$ and $R_2$ so that $R_1 \cup R_2$ is a splitting of $R$. Alternatively, we may first affinely transform $R_1$ and $R_2$ to make them sufficiently flat.

Let $c$ be a $k$-coloring of~$R_1 \times R_2$.
Let $z_1 < \cdots < z_{|R_1|}$ be the ordering of the points of~$R_1$ according to their increasing $x$\nobreakdash-coordinates.
We assign a vector $v(x) \in [k]^{|R_1|}$ to each $x \in R_2$, where the $i$th coordinate of $v(x)$ is the color of the pair $(z_i,x) \in R_1 \times R_2$ in $c$.
Note that the number of distinct vectors $v(x)$ is at most $k^{|R_1|}$.
Let $c'$ be the $k^{|R_1|}$-coloring of the points from $R_2$ obtained by coloring each point $x \in R_2$ with $v(x)$.
It follows from the choice of $R_2$ that there is a subset $S'_2$ of $R_2$ such that $S'_2$ and $S_2$ have the same signature and all points of $S'_2$ have the same color in $c'$.

Now, for every point $z$ from $R_1$, all pairs $(z,x)$ with $x \in S'_2$ have the same color $i_z$ from $[k]$.
Let $c''$ be the $k$-coloring of the points of $R_1$ where every point $z \in R_1$ has the color $i_z$.
By the choice of $R_1$ there is a subset $S'_1$ of $R_1$ such that $S'_1$ and $S_1$ have the same signature and $S'_1$ is monochromatic in $c''$.  
It follows from the choice of $c''$ that $S'_1 \times S'_2$ is monochromatic in $c$.
\end{proof}

We are now ready to prove Theorem~\ref{thm-decomposableRamseyNondiagonal}.

\begin{proof}[Proof of Theorem~\ref{thm-decomposableRamseyNondiagonal}]
Let $k$ be a positive integer and let $Q_1,\dots,Q_k$ be decomposable point sets.
We proceed by induction on $|Q_1|+\dots+|Q_k|$ and we find a point set $P$ with $P \xrightarrow{{x}} (Q_1,\dots,Q_k)^2$.
We assume that $k \ge 2$, as otherwise we can choose $P = Q_1$.

If there is an $i \in [k]$ with $|Q_i|=1$, then any non-empty point set $P$ satisfies $P \xrightarrow{{x}} (Q_1,\dots,Q_k)^2$.
This constitutes the base case.

For the induction step, we thus assume $|Q_1|,\dots,|Q_k| \ge 2$.
For every $i \in [k]$, let $Q^1_i$ and $Q^2_i$ be two nonempty disjoint subsets of $Q_i$ such that $Q^1_i \cup Q^2_i$ is a splitting of $Q_i$ and $Q^1_i$ and $Q^2_i$ are both decomposable.
The parts $Q^1_i$ and $Q^2_i$ exist, since $Q_i$ is decomposable and $|Q_i| \ge 2$.

For every $i \in [k]$, we let $T_i$ be a point set such that 
\[T_i \xrightarrow{{x}} (Q_1,\dots,Q_{i-1},Q^1_i,Q_{i+1},\dots,Q_k)^2.\]
The sets $T_i$ exist by the induction hypothesis, since 
\[|Q_1|+\dots+|Q_{i-1}|+|Q_i^1|+|Q_{i+1}|+\cdots+|Q_k| < |Q_1|+\cdots+|Q_k|\]
and all the sets $Q_1,\dots,Q_{i-1},Q^1_i,Q_{i+1},\dots,Q_k$ are decomposable.
Similarly, for every $i \in [k]$, we let $U_i$ be a point set such that 
\[U_i \xrightarrow{{x}} (Q_1,\dots,Q_{i-1},Q^2_i,Q_{i+1},\dots,Q_k)^2.\]
Again, the sets $U_i$ exist by the induction hypothesis.

Let $S_1$ be a disjoint union $T_1 \cup \cdots \cup T_k$ and $S_2$ be a disjoint union $U_1 \cup \cdots \cup U_k$ such that $S_1$ and $S_2$ are both in general position.
Let $S$ be a point set obtained by translating and scaling $S_1$ and $S_2$ so that $S_1 \cup S_2$ is a splitting of $S$.
We apply Lemma~\ref{lem-nesValtr} to $S=S_1 \cup S_2$ and obtain a point set $P= P(S,k)$ with a splitting $P=R_1 \cup R_2$ such that 
\begin{enumerate}[label=(\roman*)]
\item $P$ is in general position,
\item for every $k$-coloring $c$ of $R_1 \times R_2$ there exists a subset $S'$ of $P$ such that $(S' \cap R_i)$ and $S_i$ have the same signature for both $i \in \{1,2\}$ and $(S'\cap R_1) \times (S' \cap R_2)$ is monochromatic in $c$.
\end{enumerate}

Let $c$ be a $k$-coloring of $\binom{P}{2}$.
By the definition of $S_1$ and $S_2$, there is a color $j \in [k]$ and sets $T'_i \subseteq R_1$ and $U'_i \subseteq R_2$, for each $i \in [k]$, such that $T'_i$ and $T_i$ have the same signature, $U'_i$ and $U_i$ have the same signature, and all pairs from $T'_i \times U'_i$ have color $j$ in $c$.

For every $i \in [k]$, the definition of $T_i$ implies that there is an $a_i \in [k]$ and a subset $A_i$ of $T'_i$ such that $A_i$ has the same signature as $Q_{a_i}$ if $a_i \neq i$ and as $Q^1_{a_i}$ if $a_i=i$ and, moreover, all pairs of points from $A_i$ have color $a_i$ in $c$.
Similarly, the definition of $U_i$ implies that there is $b_i \in [k]$ and a subset $B_i$ of $U'_i$ such that $B_i$ has the same signature as $Q_{b_i}$ if $b_i \neq i$ and as $Q^2_{b_i}$ if $b_i=i$ and, moreover, all pairs of points from $B_i$ have color $b_i$ in $c$.
We may assume that $a_i=i=b_i$ for every $i \in [k]$, as otherwise we have some $l \in [k]$ and a subset of $P$ with the same signature as $Q_l$ and with all pairs of points of color $l$ in $c$ and we are done.

Thus, for every $i \in [k]$, we have a set $Q'_i\subseteq T'_i$ with the same signature as $Q^1_i$ and a set $Q''_i\subseteq U'_i$ with the same signature as $Q^2_i$ such that all pairs of points from $Q'_i$ and all pairs of points from $Q''_i$ have color $i$ in $c$.
Since all pairs from $T'_i \times U'_i$ have color $j$ in $c$ for every $i \in [k]$, the sets $Q'_j$ and $Q''_j$ together give a set $Q'$  with all pairs from $\binom{Q'}{2}$ of color $j$ in $c$.

Since $Q'_j\cup Q''_j$ is a splitting of $Q'$, the set $Q'$ has the same signature as $Q_j$.
This finishes the proof.
\end{proof}

In the proof of Theorem~\ref{thm-decomposableRamseyNondiagonal} we use the following important property: if $A\cup B$ is a splitting and $A'\subseteq A$, $B'\subseteq B$, then $A'\cup B'$ is also a splitting. Thus, decomposable point sets form a maximal class of point sets such that all their subsets with at least two points have a nontrivial splitting. Therefore, generalizing Theorem~\ref{thm-decomposableRamseyNondiagonal} to a larger class of point sets seems to require new ideas.

%========================================================================
\section{A binary point-set predicate encoding all order ty\-pes}
\label{sec-FelsnerPredicate}

Here, using a result of Felsner~\cite{felsner97}, we construct a binary predicate that encodes order types of all sets from $\mathcal{P}$.
We note that such a predicate can also be obtained by a probabilistic argument.
This follows from the well-known fact that the number of automorphism-free graphs on $n$ vertices is $2^{\Theta(n^2)}$ while the number of different order types of sets with $n$ points is only $2^{\Theta(n\log{n})}$~\cite{alon86,goodmanPollack86}.
However, we give a specific example of such a binary predicate~$\Psi$ to also provide an example of a ``non-standard'' point set predicate.
Additionally, $\Psi$ can be used to encode order types of all generalized point sets, for which the probabilistic argument becomes slightly more complicated, as their number is also in $2^{\Theta(n^2)}$~\cite{goodmanPollack83,knuth92}.

First, we need some definitions.
An arrangement of lines is \emph{simple} if no three lines from this arrangement intersect in a common point and no two lines are parallel.
An arrangement of lines partitions the plane into \emph{faces} of dimensions 0, 1, and 2.
Incidences between faces of different dimensions naturally determine a partially ordered set, which is called the \emph{face lattice} of the arrangement.
It is well-known that sets of points in general position have a simple dual line arrangement, and that the face lattice of the arrangement determines the order type of the primal point set.
Hence, by reconstructing the dual line arrangement of a point set, we reconstruct its order type.

Let $\mathcal{A}$ be a simple arrangement of nonvertical lines $l_1,\dots,l_n$ labeled according to their decreasing slopes and oriented from left to right.
We define the following mapping $\psi_{\mathcal{A}} \colon \mathcal{A} \times [n-1] \to \{0,1\}$ for $\mathcal{A}$.
For every $i \in [n]$ and $j \in [n-1]$, let $\psi_{\mathcal{A}}(l_i,j)=1$ if the $j$th crossing along the line $l_i$ is a crossing with a line $l_k$ such that $k>i$.
Otherwise let $\psi_{\mathcal{A}}(l_i,j)=0$.

Let $P$ be a set from $\mathcal{P}$ with $|P|=n$.
We use $\delta$ to denote the duality transform that maps a point $(a,b) \in \mathbb{R}^2$ to the line $y=ax-b$.
Let $\mathcal{A}(P)$ be the dual line arrangement of $P$ obtained by $\delta$. 
Since $P$ is in general position and no two points from $P$ have the same $x$\nobreakdash-coordinate, the arrangement $\mathcal{A}(P)$ is simple.

We use the function $\psi_{\mathcal{A}(P)}$ to define a function $\Psi_P:(P)_2 \rightarrow \{0,1\}$.
The functions $\Psi_P$, $P \in \mathcal{P}$, will determine the predicate $\Psi$ by setting $\Psi(P,(p,q))=\Psi_P(p,q)$.
Let $\iota \colon P \to [n]$ be the mapping such that $\iota(p)=i$ if $\delta(p)=l_i$.
Note that $\iota$ is a one-to-one correspondence.
For distinct points $p$ and $q$ from $P$, we set 
\[\Psi_P(p,q) =
\begin{cases}
\psi_{\mathcal{A}(P)}(\delta(p),\iota(q)) & \text{if } \iota(p) >\iota(q),\\
\psi_{\mathcal{A}(P)}(\delta(p),\iota(q)-1) & \text{if } \iota(p) < \iota(q).  
\end{cases}\]

We show that $\Psi$ encodes the order type of $P$.
Let $p$ be a point from $P$ and let $l_i$ be the line from $\mathcal{A}(P)$ such that  $\delta(p)=l_i$.
The number of crossings of $l_i$ with $l_j$, $j>i$, is exactly $n-i=n-\iota(p)$.
Since the number of such crossings also equals $\sum_{j=1}^{n-1} \psi_{\mathcal{A}(P)}(\delta(p),j)=\sum_{q \in P \setminus\{p\}}\Psi_P(p,q)$, the value $\iota(p)$ is exactly $n-\sum_{q \in P \setminus\{p\}}\Psi_P(p,q)$.
Thus we can find the one-to-one correspondence $\iota$ using $\Psi_P$.
With $\iota$ we can easily recover the function $\psi_{\mathcal{A}(P)}$ from~$\Psi_P$.
The rest follows from Felsner's algorithm~\cite{felsner97}, which finds the face lattice of the line arrangement $\mathcal{A}(P)$ that is encoded by~$\psi_{\mathcal{A}(P)}$.

However, the predicate $\Psi$ is not locally consistent. 
This is because the second parameter $q$ in $\Psi_P(p,q)$ is not really related to the point $q$, since the $\iota(q)$th (or $(\iota(q)-1)$st if $\iota(p)<\iota(q)$) crossing on $\delta(p)$ might not be with the line $\delta(q)=L_{\iota(q)}$. 
The parameter $\iota(q)$ serves merely as an index of some crossing on the line $\delta(p)$.

%========================================================================
\section{Proof of Theorem~\ref{thm-noHereditaryPredicate}}
\label{sec-noHereditaryPredicate}

In this section we show that there are point sets on which no binary point-set predicate is locally consistent.

Let $G=(V,E)$ be a graph.
We say that $G$ is \emph{partially oriented} if each edge $e=\{u,v\}$ from $E$ either has no orientation or $e$ is oriented from $u$ to $v$ or from $v$ to $u$.
For a positive integer $k$, a \emph{$k$-edge-coloring of $G$} is a $k$-coloring of $E$.

The proof of Theorem~\ref{thm-noHereditaryPredicate} proceeds as follows.
If $\Gamma$ is a binary point-set predicate with codomain of size $k$ and $P \in \mathcal{P}$, we represent the function $\Gamma_P$ by a partially oriented graph $G(P,\Gamma)$, which is obtained by orienting some edges of $G_P = (P, \binom{P}{2})$, and by a certain $K$-edge-coloring $c(P,\Gamma)$ of~$G_P$, where $K = \binom{k+1}{2}$.
Then we show that if $\Gamma$ is locally consistent on~$P$, then every subgraph of $G(P,\Gamma)$ that is monochromatic in $c(P,\Gamma)$ avoids certain partially oriented subgraphs induced by four points in nonconvex position.
We find a decomposable set $S$ such that if $S$ induces a monochromatic subgraph of $G(P,\Gamma)$ in $c(P,\Gamma)$, then it contains some of the forbidden partially oriented subgraphs.
Finally, using Theorem~\ref{thm-decomposableRamsey}, we choose $P$ to be a point set such that $P \to (S)^2_{K}$, so that there is a monochromatic copy of $S$ in every $K$-edge-coloring of $G_P$.

Let $k$ be a fixed positive integer and let $Z$ be a set of size $k$.
Let $\prec$ be an arbitrary total order on~$Z$.
Let $\Gamma$ be a binary predicate with codomain $Z$ and let $P$ be a set from $\mathcal{P}$ of size $n$.

The function $\Gamma_P$ can be represented by a $K$-edge-coloring of a partially oriented graph in the following way.
Every edge $e=\{u,v\}$ of $G_P$ is oriented from $u$ to $v$ in $G(P,\Gamma)$ if $\Gamma_P(u,v) \prec \Gamma_P(v,u)$, and it is not oriented if $\Gamma_P(u,v)=\Gamma_P(v,u)$.
Let $c(P,\Gamma)$ be the $K$-edge-coloring of $G_P$ that assigns the color $\{\Gamma_P(u,v),\Gamma_P(v,u)\}$ to every edge of $G_P$  with vertices $u$ and $v$.

Clearly, given the partially oriented graph $G(P,\Gamma)$ and the edge-coloring $c(P,\Gamma)$, we can recover the function $\Gamma_P$.

Let $H$ be a partially oriented graph with vertex set $\{v_0,\dots,v_{n-1}\}$.
A vertex $v$ of $H$ is called a \emph{source in $H$} if all edges of $H$ containing $v$ are oriented from $v$.
Similarly, a vertex $u$ of $H$ is called a \emph{sink in $H$} if all edges of $H$ containing $u$ are oriented towards $u$.
We say that $H$ is an \emph{oriented cycle with orientation $(v_0,\dots,v_{n-1})$} if $H$ is a cycle with edges $\{v_i,v_{i+1}\}$ for every $i \in \{0,\dots,n-1\}$ (indices taken modulo $n$) and every edge $\{v_i,v_{i+1}\}$ is oriented from $v_i$ to $v_{i+1}$.

The following lemma captures a crucial property of the graph $G(P,\Gamma)$ and the coloring $c(P,\Gamma)$.

\begin{lemma}
\label{lem-predicateGraphColoring}
Let $H$ be an induced partially oriented subgraph of $G(P,\Gamma)$ that is monochromatic in $c(P,\Gamma)$.
If $H$ contains triangles $T_1$ and $T_2$ with distinct vertex sets $\{a_1,a_2,a_3\}$ and $\{b_1,b_2,b_3\}$, respectively, such that $\Delta_P(a_1,a_2,a_3)\neq \Delta_P(b_1,b_2,b_3)$, then the following conditions are satisfied.
\begin{enumerate}[label=(\roman*)]
\item\label{item-predicateGraph1}If no edge of $H$ is oriented, then $\Gamma$ is not locally consistent on $P$.
\item\label{item-predicateGraph2}If $T_1$ and $T_2$ are oriented cycles with orientations $(a_1,a_2,a_3)$ and $(b_1,b_2,b_3)$, respectively, then $\Gamma$ is not locally consistent on $P$.
\item\label{item-predicateGraph3}If $a_1$ is a source in $T_1$, $a_3$ is a sink in $T_1$, $b_1$ is a source in $T_2$, and $b_3$ is a sink in $T_2$, then $\Gamma$ is not locally consistent on $P$.
\end{enumerate}
\end{lemma}
\begin{proof}
Since $H$ is monochromatic in $c(P,\Gamma)$, it follows from the choice of $c(P,\Gamma)$ and $G(P,\Gamma)$ that either all the edges of~$H$ are oriented or none of them is.
In the latter case, there is an element $z \in Z$ such that $\Gamma_P(a_i,a_j)=z=\Gamma_P(b_i,b_j)$ for all distinct $i$ and $j$ from $\{1,2,3\}$.
Since $\Delta_P(a_1,a_2,a_3)\neq \Delta_P(b_1,b_2,b_3)$, we see that $\Gamma$ is not locally consistent on $P$.
This establishes part~\ref{item-predicateGraph1}.

We thus assume that all the edges in $H$ are oriented.
If $T_1$ and $T_2$ are oriented cycles with orientations $(a_1,a_2,a_3)$ and $(b_1,b_2,b_3)$, respectively, then it follows that there are two elements $z_1$ and $z_2$ from $Z$ such that $z_1 \prec z_2$ and 
\[\Gamma_P(a_1,a_2)=\Gamma_P(a_2,a_3)=\Gamma_P(a_3,a_1)=z_1=\Gamma_P(b_1,b_2)=\Gamma_P(b_2,b_3)=\Gamma_P(b_3,b_1) \]
and
\[\Gamma_P(a_2,a_1)=\Gamma_P(a_3,a_2)=\Gamma_P(a_1,a_3)=z_2=\Gamma_P(b_2,b_1)=\Gamma_P(b_3,b_2)=\Gamma_P(b_1,b_3) .\]
Again, since $\Delta_P(a_1,a_2,a_3)\neq \Delta_P(b_1,b_2,b_3)$, we see that $\Gamma$ is not locally consistent on $P$ and part~\ref{item-predicateGraph2} follows.

Finally, we assume that the assumptions in part\ref{item-predicateGraph3} are met.
The choice of $c(P,\Gamma)$ and $G(P,\Gamma)$ then implies that there are $z_1,$ and $z_2$ from $Z$ such that $z_1 \prec z_2$ and
\[\Gamma_P(a_1,a_2)=\Gamma_P(a_2,a_3)=\Gamma_P(a_1,a_3)=z_1=\Gamma_P(b_1,b_2)=\Gamma_P(b_2,b_3)=\Gamma_P(b_1,b_3) \]
and
\[\Gamma_P(a_2,a_1)=\Gamma_P(a_3,a_2)=\Gamma_P(a_3,a_1)=z_2=\Gamma_P(b_2,b_1)=\Gamma_P(b_3,b_2)=\Gamma_P(b_3,b_1) .\]
Thus the predicate $\Gamma$ is not locally consistent on $P$, as $\Delta_P(a_1,a_2,a_3)\neq \Delta_P(b_1,b_2,b_3)$.
\end{proof}

The following lemma says that if $\Gamma$ is locally consistent on $P$ then every 4-tuple of points that is not in convex position and that induces a monochromatic subgraph of $G_P$ in $c(P,\Gamma)$ admits only four specific orientations in $G(P,\Gamma)$.
We will use this to reduce the cases to be considered for a monochromatic subgraph induced by five points.

\begin{lemma}
\label{lem-quadrupleOrientation}
Let $Q$ be a subset of $P$ such that $Q=\{q_1,q_2,q_3,q_4\}$ and $Q$ has exactly three extremal points $q_1$, $q_2$, and $q_3$.
Let $H$ be the partially oriented subgraph of $G(P,\Gamma)$ induced by $Q$.
If $\Gamma$ is locally consistent on $P$ and $H$ is monochromatic in $c(P,\Gamma)$, then the set $\{q_1,q_2,q_3\}$ induces an oriented triangle in $H$ and $q_4$ is either a source or a sink in $H$.
\end{lemma}
\begin{proof}
Assume that $\Gamma$ is locally consistent on $P$.
Since $H$ is monochromatic in $c(P,\Gamma)$, part~\ref{item-predicateGraph1} of Lemma~\ref{lem-predicateGraphColoring} implies that all the edges of $H$ are oriented.

If the set $\{q_1,q_2,q_3\}$ induces an oriented triangle $T$ in $H$, then $q_4$ is either a source or a sink in $H$.
Otherwise, we assume without loss of generality that the triangle $T$ has orientation $(q_1,q_2,q_3)$, the edge $\{q_1,q_4\}$ is oriented from $q_1$ to $q_4$, and $\{q_2,q_4\}$ from $q_4$ to $q_2$.
Now, if the edge $\{q_3,q_4\}$ is oriented from $q_4$ to $q_3$, then part~\ref{item-predicateGraph3} of Lemma~\ref{lem-predicateGraphColoring} applied to the triples $(q_1,q_4,q_2)$ and $(q_4,q_2,q_3)$ implies that $\Gamma$ is not locally consistent on $P$; see part~(a) of Figure~\ref{fig-quadrupleOrientation}.
On the other hand, if the edge $\{q_3,q_4\}$ is oriented from $q_3$ to $q_4$, then part~\ref{item-predicateGraph3} of Lemma~\ref{lem-predicateGraphColoring} applied to the triples $(q_1,q_4,q_2)$ and $(q_3,q_1,q_4)$ again implies that $\Gamma$ is not locally consistent on $P$; see part~(b) of Figure~\ref{fig-quadrupleOrientation}.

\begin{figure}
\begin{center}
\includegraphics{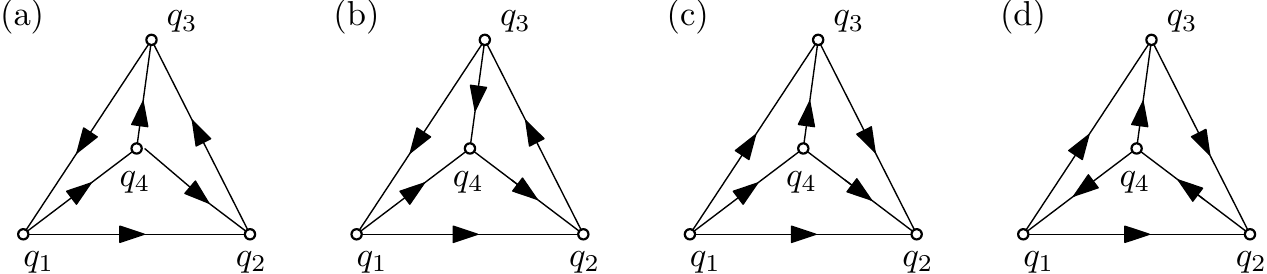}
\caption{Orientations of the graph $H$ showing that $\Gamma$ is not locally consistent on $P$.}
\label{fig-quadrupleOrientation}
\end{center}
\end{figure}

Now we show that if the set $\{q_1,q_2,q_3\}$ does not induce an oriented triangle in $H$, then $\Gamma$ is not locally consistent on $P$.
Without loss of generality, we assume that $q_1$ is a source and $q_2$ is a sink in the subgraph of $H$ induced by $\{q_1,q_2,q_3\}$.

First, suppose that the edge $\{q_2,q_4\}$ is oriented from $q_4$ to $q_2$.
The edge $\{q_1,q_4\}$ is oriented from $q_1$ to $q_4$ by part~\ref{item-predicateGraph3} of Lemma~\ref{lem-predicateGraphColoring} applied to the triples $(q_1,q_3,q_2)$ and $(q_4,q_1,q_2)$.
Part~\ref{item-predicateGraph3} of Lemma~\ref{lem-predicateGraphColoring} applied to the triples $(q_1,q_3,q_2)$ and $(q_3,q_4,q_2)$ implies that the edge $\{q_3,q_4\}$ is oriented from $q_4$ to $q_3$; see part~(c) of Figure~\ref{fig-quadrupleOrientation}.
However, then part~\ref{item-predicateGraph3} of Lemma~\ref{lem-predicateGraphColoring} applied to the triples $(q_1,q_3,q_2)$ and $(q_1,q_4,q_3)$ shows that $\Gamma$ is not locally consistent on $P$. 

It remains to consider the case when $\{q_2,q_4\}$ is oriented from $q_2$ to $q_4$.
The edge $\{q_1,q_4\}$ is oriented from $q_4$ to $q_1$ by part~\ref{item-predicateGraph3} of Lemma~\ref{lem-predicateGraphColoring} applied to the triples $(q_1,q_3,q_2)$ and $(q_1,q_2,q_4)$.
In particular, $\{q_1,q_2,q_4\}$ induces an oriented triangle with orientation $(q_1,q_2,q_4)$.
By part~\ref{item-predicateGraph2} of Lemma~\ref{lem-predicateGraphColoring} applied to the triples $(q_1,q_2,q_4)$ and $(q_1,q_3,q_4)$, the edge $\{q_3,q_4\}$ is oriented from $q_4$ to $q_3$; see part~(d) of Figure~\ref{fig-quadrupleOrientation}.
Then, however,  part~\ref{item-predicateGraph2} of Lemma~\ref{lem-predicateGraphColoring} applied to the triples $(q_1,q_2,q_4)$ and $(q_4,q_3,q_2)$ implies that $\Gamma$ is not locally consistent on~$P$.

Altogether, we see that $H$ admits only the following four orientations:
the set $\{q_1,q_2,q_3\}$ induces an oriented triangle in $H$ and $q_4$ is either a source or a sink in $H$.
\end{proof}

\begin{observation}
\label{obs-decomposable5Tuples}
There is a 5-tuple $S$ of points that has three extremal points and that  is decomposable.
\end{observation}
\begin{proof}
It suffices to find a recursive decomposition of some 5-tuple of points with three extremal points such that the decomposition satisfies the conditions in the definition of decomposable sets.
Such a set and the first two steps of this decomposition are illustrated in Figure~\ref{fig-decomposable5Tuple}.
\end{proof}

\begin{figure}
\begin{center}
\includegraphics{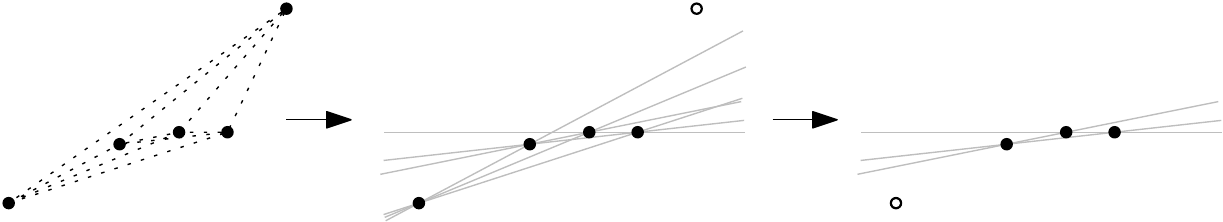}
\caption{A decomposable point set of size $5$ with three extremal points.
The parts of the splitting are distinguished by black and white in every step.}
\label{fig-decomposable5Tuple}
\end{center}
\end{figure}

We are now ready to prove Theorem~\ref{thm-noHereditaryPredicate}.
Let $S$ be a decomposable set of five points with three extremal points and let $P$ be a point set such that $P \to (S)^2_{K}$. 
The sets $S$ and~$P$ exist by Observation~\ref{obs-decomposable5Tuples} and Theorem~\ref{thm-decomposableRamsey}.

Suppose that $\Gamma$ is a binary predicate with a codomain of size $k$.
By the choice of~$P$, there is a subset $R$ of~$P$ such that $R$ and $S$ have the same order type and $\binom{R}{2}$ is monochromatic in $c(P,\Gamma)$.
Let $r_1,r_2,r_3,r_4,r_5$ be the points in $R$ such that $r_3$ and $r_4$ are in the interior of the convex hull of $R$, the line $\overline{r_3r_4}$ separates $r_5$ from $r_1$ and $r_2$, and the line $\overline{r_1r_3}$ separates $r_5$ from $r_2$ and $r_4$; see Figure~\ref{fig-5TupleOrientation}.
For $i\in\{1,2,3,4\}$, let $Q_i = R \setminus\{r_i\}$ and note that each $Q_i$ has the same order type as the set $Q$ from Lemma~\ref{lem-quadrupleOrientation}.

Consider the partially oriented subgraph $F$ of $G(P,\Gamma)$ induced by $R$ and, for every $i \in \{1,2,3,4\}$, let $H_i$ be the partially oriented subgraph of $F$ induced by~$Q_i$.
We show that at least one of the graphs $H_i$ has none of the allowed orientations.

\begin{figure}
\centering
\includegraphics{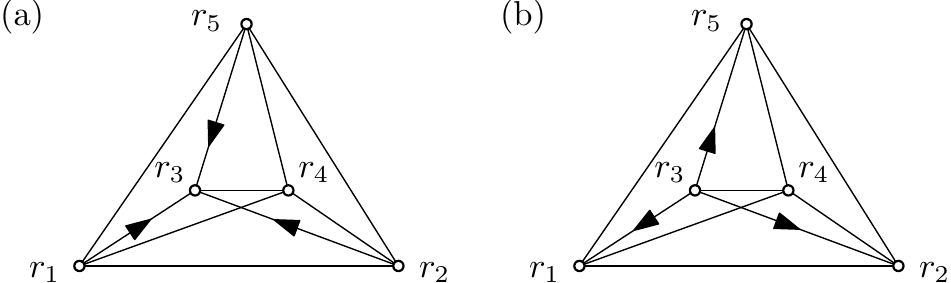}
\caption{Orientations of the graph $F$ showing that $\Gamma$ is not locally consistent on $P$.}
\label{fig-5TupleOrientation}
\end{figure}

By Lemma~\ref{lem-quadrupleOrientation} applied to $H_4$, the point $r_3$ is either a source or a sink in $H_4$. But then the set $\{r_2,r_3,r_5\}$ does not induce an oriented triangle in $H_1$, so by Lemma~\ref{lem-quadrupleOrientation} applied to~$H_1$, the predicate $\Gamma$ is not locally consistent on $P$; see Figure~\ref{fig-5TupleOrientation}. 

This finishes the proof of Theorem~\ref{thm-noHereditaryPredicate}.

%========================================================================
\section{Proof of Proposition~\ref{prop-hereditaryPredicate}}
\label{sec-hereditaryPredicate}

We construct a binary predicate $\Phi$ with codomain $\{-1,1\}$ that is locally consistent on the class $\mathcal{W}$ of all wheel sets and that encodes their order types.

For a wheel set $P$, let $w_P$ be the leftmost point of $P$ if the set $P$ is in convex position and let $w_P$ be the unique point of $P$ in the interior of the convex hull of $P$ otherwise.
We define the predicate $\Phi$ by setting
\[\Phi_P(p,q) =
\begin{cases}
-1 & \text{if } p=w_P,\\
1 & \text{if } q=w_P,\\
\Delta_P(p,q,w_P) & \text{otherwise.}
\end{cases}
\]
for every $P \in \mathcal{W}$ and every pair $(p,q) \in (P)_2$.

We first show that $\Phi$ encodes order types of all sets from $\mathcal{W}$.
Let $P=\{p_1,\dots,p_n\}$ and $Q=\{q_1,\dots,q_n\}$ be two wheel sets and let $f \colon P \to Q$ be a one-to-one correspondence such that $\Phi_P(p_i,p_j)=\Phi_Q(f(p_i),f(p_j))$ for all distinct $i$ and $j$ from $([n])_2$.
For a point $p \in P$, we have $\Phi_P(p,q)=-1$ and $\Phi_P(q,p)=1$ for every $q \in P \setminus \{p\}$ if and only if $p=w_P$. 
An analogous statement is true for $w_Q$ and $\Phi_Q$.
Thus we have $f(w_P)=f(w_Q)$.
Let $(p_i,p_j,p_k)$ be a triple from $(P)_3$.
Assume first that $w_P \in \{p_i,p_j,p_k\}$.
Without loss of generality, we assume $w_P=p_k$, as otherwise we proceed analogously.
Then $\Delta_P(p_i,p_j,p_k) = \Phi_P(p_i,p_j)=\Phi_Q(f(p_i),f(p_j))=\Delta_Q(f(p_i),f(p_j),f(p_k))$ and the triples $(p_i,p_j,p_k)$ and $(f(p_i),f(p_j),f(p_k))$ have the same orientation.

Assume $w_P \notin \{p_i,p_j,p_k\}$.
Observe that $\Delta_P(p_i,p_j,p_k)=1$ if and only if $\Phi_P(p_i,p_j) + \Phi_P(p_j,p_k) + \Phi_P(p_k,p_i) \ge 1$.
In particular, if $\Phi_P(p_i,p_j) + \Phi_P(p_j,p_k) + \Phi_P(p_k,p_i) = 3$, then $w_P$ is in the interior of the triangle $p_ip_jp_k$ and if $\Phi_P(p_i,p_j) + \Phi_P(p_j,p_k) + \Phi_P(p_k,p_i) = 1$, then $\{p_i,p_j,p_k,w_P\}$ is in convex position.
Similarly, $\Delta_P(p_i,p_j,p_k)=-1$ if and only if $\Phi_P(p_i,p_j) + \Phi_P(p_j,p_k) + \Phi_P(p_k,p_i) \le -1$.
Analogous statements are true for $\Phi_Q$ and $\Delta_Q$.
The assumption $\Phi_P(p_r,p_s)=\Phi_Q(f(p_r),f(p_s))$ for all distinct $r$ and $s$ from $([n])_2$ thus implies that the triples $(p_i,p_j,p_k)$ and $(f(p_i),f(p_j),f(p_k))$ have the same orientation.
Consequently, the point sets $P$ and $Q$ have the same order type and $\Phi$ encodes order types of all sets from $\mathcal{W}$.

We now show that $\Phi$ is locally consistent on $\mathcal{W}$.
Let $P$ be a wheel set and let $\{a_1,a_2,a_3\}$ and $\{b_1,b_2,b_3\}$ be two triples of points from $P$ with $\Phi_P(a_i,a_j)=\Phi_P(b_i,b_j)$ for all distinct $i$, $j$ from $\{1,2,3\}$.
Assume that $w_P=a_k$ for some $k \in \{1,2,3\}$ and $w_P \in \{b_1,b_2,b_3\}$.
Then $w_P=b_k$, because $b_k$ is the only point from $\{b_1,b_2,b_3\}$ with $\Phi_P(b,b_k)=-1$ for every $b \in \{b_1,b_2,b_3\}\setminus\{b_k\}$.
We let $i$ and $j$ be integers such that $\{i,j\} = \{1,2,3\} \setminus \{k\}$ and $i<j$.
It follows from the definition of $\Phi$ that $\Delta_P(a_i,a_j,a_k) = \Phi_P(a_i,a_j)=\Phi_P(b_i,b_j) = \Delta_P(b_i,b_j,b_k)$.
Consequently, $\Delta_P(a_1,a_2,a_3)=\Delta_P(b_1,b_2,b_3)$.

Assume that $w_P=a_k$ for some $k \in \{1,2,3\}$ and $w_P \notin \{b_1,b_2,b_3\}$.
Let $i$ and $j$ be integers such that $\{i,j\} = \{1,2,3\} \setminus \{k\}$ and $i<j$.
Then $\Delta_P(a_i,a_j,a_k) = \Phi_P(a_i,a_j) = \Phi_P(b_i,b_j)$ and it suffices to show that $\Phi_P(b_i,b_j) = \Delta_P(b_i,b_j,b_k)$.
Since $1=\Phi_P(a_i,a_k)=\Phi_P(b_i,b_k)=\Delta_P(b_i,b_k,w_P)$ and $1=\Phi_P(a_j,a_k)=\Phi_P(b_j,b_k)=\Delta_P(b_j,b_k,w_P)$, the points $b_i$ and $b_j$ lie on the same side of the line $\overline{b_kw_P}$.
If $\Phi_P(b_i,b_j) \neq \Delta_P(b_i,b_j,b_k)$, then, since $\Phi_P(b_i,b_j)=\Delta_P(b_i,b_j,w_P)$, the points $b_k$ and $w_P$ lie in opposite halfplanes determined by $\overline{b_ib_j}$.
However, then $b_i$ or $b_j$ lies in the interior of the convex hull of $\{b_j,b_k,w_P\}$ or $\{b_i,b_k,w_P\}$, respectively, which is impossible as $w_P$ is the only interior point of $P$ (if there is any) and $w_P \notin \{b_i,b_j,b_k\}$.
Therefore $\Phi_P(b_i,b_j) = \Delta_P(b_i,b_j,b_k)$, which implies $\Delta_P(a_1,a_2,a_3) = \Delta_P(b_1,b_2,b_3)$.
By symmetry, if $w_P \in \{b_1,b_2,b_3\}$ and $w_P \notin \{a_1,a_2,a_3\}$, then $\Delta_P(a_1,a_2,a_3)=\Delta_P(b_1,b_2,b_3)$.

It remains to deal with the case $w_P \notin \{a_1,a_2,a_3\}$ and $w_P \notin \{b_1,b_2,b_3\}$.
Then we have $\Delta_P(a_i,a_j,w_P)=\Phi_P(a_i,a_j)=\Phi_P(b_i,b_j)=\Delta_P(b_i,b_j,w_P)$ for all distinct $i$, $j$ from $\{1,2,3\}$.
Again, we have $\Delta_P(a_1,a_2,a_3)=1$ if and only if $\Phi_P(a_1,a_2) + \Phi_P(a_2,a_3) + \Phi_P(a_3,a_1) \ge 1$ and, similarly, $\Delta_P(a_1,a_2,a_3)=-1$ if and only if $\Phi_P(a_1,a_2) + \Phi_P(a_2,a_3) + \Phi_P(a_3,a_1) \le -1$.
Since analogous claims hold also for $\Delta_P(b_1,b_2,b_3)$, it follows that $\Delta_P(a_1,a_2,a_3) = \Delta_P(b_1,b_2,b_3)$.
Thus $\Phi$ is locally consistent on $P$, which finishes the proof.

%-----------------------------------------------------------------------------
\subsection{Predicates assigning tournaments to order types}
Observe that the predicate $\Phi$ from Proposition~\ref{prop-hereditaryPredicate} satisfies $\Phi_P(a,b) = - \Phi_P(b,a)$ for every wheel set $P$ and every pair $(a,b) \in (P)_2$.
Hence, $\Phi$ defines a tournament on $P$; that is, an orientation of the complete graph on~$P$.
We show that wheel sets are the only point sets with a locally consistent predicate of this form.

\begin{proposition}
Let $P$ be a point set and let $\Gamma$ be a binary point-set predicate with codomain $\{-1,1\}$ such that $\Gamma$ is locally consistent on $P$ and such that $\Gamma_P(a,b) = -\Gamma_P(b,a)$ for every pair $(a,b) \in (P)_2$. Then $P$ is a wheel set.
\end{proposition}

\begin{proof}
Suppose for contrary that $P$ is a point set with at least two interior points $p$, $q$.
Then, using the notation from Section~\ref{sec-noHereditaryPredicate}, the graph $G(P,\Gamma)$ is a tournament on $P$ and the coloring $c(P,\Gamma)$ uses only a single color $\{-1,1\}$. 
In the proof of Theorem~\ref{thm-noHereditaryPredicate} it has been shown that $P$ cannot contain a five-point set with triangular convex hull.
Therefore, no triangle determined by extremal points of $P$ contains the points $p$ and $q$ in its interior. 
In particular, the line determined by the points $p$ and $q$ crosses two disjoint edges $ab$ and $cd$ of the convex hull of $P$.
Then $p$ and $q$ lie in the convex hull of $\{a,b,c,d\}$.
Without loss of generality, let $p$ be inside the triangle $(a,b,c)$, $q$ be inside the triangle $(c,d,a)$, and $\Delta_P(a,b,c) = \Delta_P(c,d,a)$.
By Lemma~\ref{lem-quadrupleOrientation}, both $(a,b,c)$ and $(c,d,a)$ are oriented triangles.
Since the triangles intersect at the edge $ac$, they have opposite orientations. This contradicts the local consistency of $\Gamma$.
\end{proof}

%========================================================================
\section{Proof of Theorem~\ref{thm-smallSets}}
\label{sec-smallSets}

Here we show that for every positive integer $k$, there is a binary predicate with codomain of size $k$ that is locally consistent on all point sets of size at most $ck^{3/2}$ for some constant $c>0$ and that encodes their order types.

For an integer $n \ge 4$, let $k = \lceil c'n^{2/3} \rceil$ for some sufficiently large constant $c'$ and let $F$ be the set of functions $f \colon ([n])_2 \to [k]$.
Clearly, $|F|= k^{n(n-1)}$.
We say that two functions $f_1$ and $f_2$ from $F$ are \emph{equivalent}, written $f_1 \sim f_2$, if there is a permutation $\pi \colon [n] \to [n]$ such that $f_1(i,j) = f_2(\pi(i),\pi(j))$ for every pair $(i,j) \in ([n])_2$.
Note that there are at most $n!$ functions in each equivalence class of $\sim$.

We recall that the number of different order types of sets with $n$ points is at most $2^{c''n\log{n}}$ for some constant $c''$~\cite{alon86,goodmanPollack86}.
We show that there is a subset $F'$ of $F$ of size at least $2^{c''n\log{n}}$ such that no two functions from $F'$ are equivalent and every function $f$ from $F'$ satisfies the following condition: there are no two distinct subsets $\{a_1,a_2,a_3\}$ and $\{b_1,b_2,b_3\}$ of $[n]$ with $f(a_i,a_j)=f(b_i,b_j)$ for all distinct $i$ and $j$ from $\{1,2,3\}$.

Since $|F'| \ge 2^{c''n\log{n}}$, we can assign distinct functions from $F'$ to distinct order types of point sets of size~$n$.
For each order type $O$, we choose a point set $R_O$ with the order type $O$ as a representative of $O$ and we let $R_O=\{r_1,\dots,r_n\}$ be an arbitrary labeling of $R_O$.
For every point set $P$ with the order type $O$ there is a one-to-one correspondence $l_P \colon P \to R_O$ such that $\Delta_P(a,b,c) = \Delta_{R_O}(l_P(a),l_P(b),l_P(c))$ for every triple $(a,b,c) \in (P)_3$.
Let $f$ be the function from $F'$ that has been assigned to $O$.
We then let $\Gamma_P(p,p') = f(i,j)$ for every pair $(p,p') \in (P)_2$ such that $l_P(p) = r_i$ and $l_P(p') = r_j$.

We show that the resulting binary point-set predicate $\Gamma$ is locally consistent on point sets of size $n$ and encodes their order types.
The local consistency is ensured by the condition on functions from $F'$, which says that no function from $F'$ attains the same values on distinct triples from $\binom{[n]}{3}$.

To show that $\Gamma$ encodes the order types of all point sets of size $n$, we use the fact that the functions from $F'$ are pairwise nonequivalent and that each function from $F'$ is ``rigid''  in the sense that it has no nontrivial automorphism acting on $[n]$.
Let $P$ and $Q$ be two point sets of size $n$ with $\Gamma_P(p,p') = \Gamma_Q(g(p),g(p'))$ for every pair $(p,p') \in (P)_2$ and some one-to-one correspondence $g \colon P \to Q$.
Let $f_P$ and $f_Q$ be the functions from $F'$ used in the definitions of $\Gamma_P$ and $\Gamma_Q$, respectively.
It follows from the definition of $\Gamma_P$ and $\Gamma_Q$ that $f_P(i,j) = f_Q(\pi(i),\pi(j))$ for every pair $(i,j) \in ([n])_2$ and some permutation $\pi$ on $[n]$.
Since no two functions from $F'$ are equivalent, we have $f_P = f = f_Q$ for some $f \in F'$ and, in particular, $P$ and $Q$ have the same order type~$O$.
Since $f$ does not attain the same values on distinct triples from $\binom{[n]}{3}$ and $n \ge 4$, the permutation $\pi$ is the identity on $[n]$.
Consequently, the function $g$ is determined uniquely and maps every $p \in P$ with $l_P(p) = r_i$ to $q \in Q$ with $l_Q(q) = r_i$, where $l_P \colon P \to R_O$ and $l_Q \colon Q \to R_O$ are the one-to-one correspondences that preserve orientations of triples and that were used in the definitions of $\Gamma_P$ and $\Gamma_Q$, respectively.
Thus $l_P(p) = l_Q(g(p))$ for every $p \in P$.
Therefore $\Delta_P(a,b,c) = \Delta_{R_O}(l_P(a),l_P(b),l_P(c)) = \Delta_{R_O}(l_Q(g(a)),l_Q(g(b)),l_Q(g(c))) =  \Delta_Q(g(a),g(b),g(c))$ for every triple $(a,b,c) \in (P)_3$ and $\Gamma$ encodes the order types of point sets of size $n$.

Thus it remains to prove the existence of the set $F'$.
We use a probabilistic approach based on the Lov\'{a}sz local lemma~\cite{alonSpencer16}.

\begin{lemma}[The Lov\'{a}sz local lemma~{\cite[Lemma~5.1.1]{alonSpencer16}}]
\label{thm-LLL}
Let $A_1,\dots,A_m$ be events in an arbitrary probability space.
Let $D=([m],E)$ be a directed graph such that for each $i$, $1 \le i \le m$, the event $A_i$ is mutually independent of all the events $A_j$ with $(i,j) \notin E$.
Suppose that there are real numbers $x_1,\dots,x_m$ such that $0 \le x_i < 1$ and $\Pr[A_i] \le x_i \prod_{(i,j) \in E}(1-x_j)$ for every $i$ with $1 \le i \le m$.
Then 
\[\Pr\left[\bigcap_{i=1}^m \overline{A_i}\right] \ge \prod_{i=1}^m\left(1-x_i\right).\]
\end{lemma}

Let $f$ be a function from $F$ chosen uniformly independently at random and let $\{a_1,a_2,a_3\}$ and $\{b_1,b_2,b_3\}$ be two distinct subsets of $[n]$.
Let $A=A_{\{(a_1,a_2,a_3),(b_1,b_2,b_3)\}}$ be the event that $f(a_i,a_j)=f(b_i,b_j)$ for all distinct $i$ and $j$ from $\{1,2,3\}$.
We say that $A$ has \emph{type 1} if  $|\{a_1,a_2,a_3\} \cap \{b_1,b_2,b_3\}| \le 1$ and \emph{type 2} if $|\{a_1,a_2,a_3\} \cap \{b_1,b_2,b_3\}| = 2$.
If $A$ has type 1, then $\Pr[A] = 1/k^6$.
Otherwise $A$ has type 2 and $\Pr[A] \le 1/k^4$.

Let $N_A$ be the set of events $A_{\{(a'_1,a'_2,a'_3),(b'_1,b'_2,b'_3)\}}$ such that $A$ is mutually independent of all the events from the complement of $N_A$ and let $N'_A$ and $N''_A$ be the subsets of $N_A$ formed by events of type 1 and 2, respectively.
Observe that $|N'_A| \le 6 \cdot 12 \cdot n^4 = 72n^4$ and $|N''_A| \le 6 \cdot 12 \cdot n^2 = 72n^2$.

Let $x_A = 1/(72n^4)$ if $A$ has type 1 and $x_A = 1/(72n^2)$ if $A$ has type 2.
Then 
\begin{align*}
x_A \prod_{A' \in N_A} (1-x_{A'}) &=  x_A\left(1-\frac{1}{72n^4}\right)^{|N'_A|}\left(1-\frac{1}{72n^2}\right)^{|N''_A|}\\ &\ge x_A\left(1-\frac{1}{72n^4}\right)^{72n^4}\left(1-\frac{1}{72n^2}\right)^{72n^2}.
\end{align*}
Since $(1-1/x)^x > 1/(2e)$ for every $x \ge 2$, the above expression is at least $x_A/(4e^2)$.
If $A$ has type 1, then the condition from Lemma~\ref{thm-LLL} is satisfied for $k \ge e^{1/3}288^{1/6}n^{2/3}$, as then $\Pr[A]=1/k^6 \le 1/(288e^2n^4)= x_A/(4e^2)$.
If $A$ has type 2, then the condition is satisfied even for $k \ge e^{1/2}288^{1/4}n^{1/2}$, as $\Pr[A] \le 1/k^4 \le 1/(288e^2n^2)= x_A/(4e^2)$.

Let $E_1$ and $E_2$ be the sets of events $A_{\{(a_1,a_2,a_3),(b_1,b_2,b_3)\}}$ of types 1 and 2, respectively.
By Lemma~\ref{thm-LLL}, the probability that none of the events $A_{\{(a_1,a_2,a_3),(b_1,b_2,b_3)\}}$ occurs is at least
\begin{align*}
\prod_{A \in E_1 \cup E_2}(1-x_A) &= \left(1-\frac{1}{72n^4}\right)^{|E_1|} \left(1-\frac{1}{72n^2}\right)^{|E_2|} \\
&\ge (1/2e)^{|E_1|/(72n^4) + |E_2|/(72n^2)}.
\end{align*}
We have $|E_1| \le n^6$ and $|E_2| \le n^3 \cdot 3 \cdot 6 \cdot n \le 18n^4$ and thus the probability is at least $(1/2e)^{n^2}$. 
Using this estimate together with the fact that there are at most $n!$ functions from $F$ in each equivalence class of $\sim$, we obtain a set $F'$ of at least $|F|/(n! \cdot (2e)^{n^2})$ functions from $F$ that satisfy the following two conditions: no two functions from $F'$ are equivalent and for every function $f$ from $F'$ there are no two distinct subsets $\{a_1,a_2,a_3\}$ and $\{b_1,b_2,b_3\}$ of $[n]$ with $f(a_i,a_j)=f(b_i,b_j)$ for all distinct $i$ and $j$ from $\{1,2,3\}$.
Since $|F| = k^{n(n-1)}$ and $k \ge c'n^{2/3}$, we also have $|F'| \ge 2^{c''n\log{n}}$ if $c'$ is sufficiently large.

%========================================================================
\section{Final remarks}
\label{sec-final}

In this section, we present several new open problems and discuss possible directions for future research.

%------------------------------------------------------------------------
\subsection{Induced Ramsey-type results for order types}
\label{subsec-finalRamseyOrderType}

Ne\v{s}et\v{r}il and Valtr~\cite{nesetrilValtr94} showed that for all integers $k,p \ge 2$ there are point sets that are not $(k,p)$-Ramsey.
We have shown that for every positive integer $k$ all decomposable point sets are $(k,2)$-Ramsey.
In fact, it is not difficult to characterize sets that are $(k,p)$-Ramsey for $k \ge 2$ and $p \ge 3$; see Proposition~\ref{prop-RamseyCharacterization}.
Given these facts, it might be interesting to see which other point sets are $(k,2)$-Ramsey and whether there is a simple characterization of $(k,2)$-Ramsey sets.

\begin{proposition}
\label{prop-RamseyCharacterization}
For all integers $k \ge 2$ and $p \ge 3$, a point set $Q$ with $|Q| > p$ is $(k,p)$-Ramsey if and only if $Q$ is in convex position.
\end{proposition}

\begin{proof}
It suffices to show that if $Q$ is $(2,p)$-Ramsey, then it is in convex position.
Assume that $Q$ is $(2,p)$-Ramsey and let $P$ be a point set with $P \to (Q)^p_2$.
Let $\lessdot$ be the lexicographic ordering on $\mathbb{R}^2$. In this proof we use only the property that $\lessdot$ is linear.
For every $i \in [p-2]$, we construct a 2-coloring $c_i$ of $\binom{P}{p}$ by coloring every $p$-tuple $a_1 \lessdot \cdots \lessdot a_p$ of points from $P$ according to the orientation of $(a_i,a_{i+1},a_{i+2})$.

Let $R$ be a subset of $P$ with $|R|>p$ and with $\binom{R}{p}$ monochromatic in $c_i$.
It follows from the definition of $c_i$ that for every $(p+1)$-tuple $r_1 \lessdot \cdots \lessdot r_{p+1}$ of points from $R$ the orientations of the triples $(r_i,r_{i+1},r_{i+2})$ and $(r_{i+1},r_{i+2},r_{i+3})$ are the same.
Since $Q$ is $(2,p)$-Ramsey, the coloring $c_i$ contains a monochromatic copy of $Q$ for every $i\in[p-2]$.
Consequently, if $q_1, \dots, q_{p+1}\in Q$ such that $q_1\lessdot \dots \lessdot q_{p+1}$, then all the $p-1$ triples $(q_1,q_2,q_3),\dots,(q_{p-1},q_p,q_{p+1})$ have the same orientation.
This implies that all triple orientations in $Q$ are the same and therefore $Q$ is in convex position.

On the other hand, let $Q$ be a set of $n$ points in convex position and let $k \ge 2$ and $p \ge 3$.
Let $R$ be a set of $r$ points in convex position for some sufficiently large integer $r=r(k,n,p)$.
Let $c$ be a coloring of $p$-tuples of points from $R$.
If $r$ is sufficiently large, then Ramsey's theorem~\cite{ramsey30} implies that $R$ contains a subset $Q'$ of size $n$ such that $\binom{Q'}{p}$ is monochromatic in $c$.
The sets $Q'$ and $Q$ have the same order type, since they are both in convex position.
\end{proof}

A slight modification of the proof of Proposition~\ref{prop-RamseyCharacterization} gives the following statement.
For all integers $k \ge 2$ and $p \ge 3$, a point set $Q=\{q_1,\dots,q_n\}$, with $x(q_1)<\dots<x(q_n)$, is ordered $(k,p)$-Ramsey if and only if all triples $(q_i,q_j,q_k)$ with $1 \le i < j < k \le n$ have the same orientation. Such sets are also often called \emph{cups} and \emph{caps}~\cite{erdosSzekeres1935}.

%------------------------------------------------------------------------
\subsection{Induced Ramsey-type results with orderings}
\label{subsec-finalRamseySignature}

Another direction for future research might be to extend the problem of determining whether a point set is ordered $(2,p)$-Ramsey to more general configurations.
Consider the following natural hypergraph variant of this problem, which can be obtained by representing a point set $P=\{p_1,\dots,p_n\}$, with $x(p_1)<\dots<x(p_n)$, by a 3-uniform hypergraph on $P$ where $\{p_i,p_j,p_k\}$, with $i<j<k$, is an edge if and only if $\Delta_P(p_i,p_j,p_k)=1$.

Let $p$ be a positive integer.
A hypergraph is \emph{ordered} if its vertex set is ordered according to some total order.
We say that an ordered 3-uniform hypergraph $H$ is \emph{ordered $(2,p)$-Ramsey} if there is an ordered 3-uniform hypergraph $G$ such that in every $2$-coloring $c$ of $p$-tuples of vertices of $G$ there is an induced ordered sub-hypergraph $H'$ in $G$ such that $H$ and $H'$ are order-isomorphic and all $p$-tuples of vertices of $H'$ have the same color in~$c$.
The following statement follows from a result of Ne\v{s}et\v{r}il and R\"{o}dl~\cite[Theorem~A]{nesRod83} and from a modification of the proof of Proposition~\ref{prop-RamseyCharacterization}.
It gives a full characterization of ordered $(2,p)$-Ramsey hypergraphs.

\begin{corollary}
\label{thm-hypergraphVariant}
Let $p$ be a positive integer and let $H$ be an ordered 3-uniform hypergraph.
If $p \le 2$, then $H$ is ordered $(2,p)$-Ramsey.
If $p \ge 3$, then $H$ is ordered $(2,p)$-Ramsey if and only if $H$ is complete or empty.
\end{corollary}

A \emph{generalized point set} $Q$ is a finite set of points together with a set of $x$-monotone curves, each pair from $\binom{Q}{2}$ contained in one such curve, so that these curves form an arrangement of pseudolines; see~\cite{felsner97} for more detailed definitions.
Analogously as for point sets, one can define a \emph{signature} for every generalized point set~\cite[Section~3.2]{bfk15}.
We then say that a generalized point set $Q$ is \emph{ordered $(2,p)$-Ramsey} if there is a generalized point set $P$ such that in every $2$-coloring $c$ of $\binom{P}{p}$ there is a subset $Q'$ of $P$ such that $Q$ and $Q'$ have the same signature and $\binom{Q'}{p}$ is monochromatic in $c$.

If $p \neq 2$, then one can use similar methods as for point sets and characterize generalized point sets that are ordered $(2,p)$-Ramsey.
However, the case of $p=2$ is wide open.

\begin{problem}
\label{prob-pseudolinear}
Is there a generalized point set that is not ordered $(2,2)$-Ramsey?
\end{problem}

Unlike in the case of hypergraphs, there are point sets that are not ordered $(2,2)$-Ramsey~\cite{nesetrilValtr94}, but the proof of this fact relies on the notion of Euclidean distance, which is not present in the case of generalized point sets.
Note that there might be a point set that is $(2,2)$-Ramsey as a generalized point set, but is not $(2,2)$-Ramsey as a point set.
It is also possible to use signatures to represent generalized point sets by ordered 3-uniform hypergraphs in an analogous way as for point sets.
It can be shown that the obtained hypergraphs are characterized by a list of 8 forbidden induced ordered sub-hypergraphs on 4 vertices~\cite[Theorem~3.2]{bfk15}.
Ne\v{s}et\v{r}il and R\"{o}dl~\cite{nesRod83} provided a sufficient condition for classes of hypergraphs with forbidden ordered sub-hypergraphs to be $(2,2)$-Ramsey.
However, this condition does not apply in our case, since we forbid induced ordered sub-hypergraphs.

%-------------------------------------------------------------------------
\subsection{Binary predicates for order types}
\label{subsec-finalPredicates}

We have seen that there are binary predicates with codomain $\{-1,1\}$ that are locally consistent on wheel sets and that encode order types of wheel sets.
On the other hand, there is no locally consistent binary point-set predicate. 
It might be interesting to find some other classes $\mathcal{C}$ of point sets for which there are binary predicates that are locally consistent on~$\mathcal{C}$ and that encode order types of all sets from $\mathcal{C}$.
We have seen that to this end, the corresponding graph must contain both directed and undirected edges.

The growth rate of the function $h$ defined in Section~\ref{subsec-introRamsey} is also unknown.
We recall that $h(k)$ is finite for every positive integer $k$ by Theorem~\ref{thm-noHereditaryPredicate} and $h(k) \ge \Omega(k^{3/2})$ by Theorem~\ref{thm-smallSets}.
We suspect that the lower bound on $h(k)$ can be improved, as the only geometric argument that we use in the proof of Theorem~\ref{thm-smallSets} is the upper bound on the number of different order types.

\paragraph{Acknowledgment}
We would like to thank to the referee for valuable comments.

\bibliographystyle{vlastni}
\bibliography{bibliography}
\end{document}